\documentclass{amsart}
\usepackage{amssymb,amsmath,amscd,enumerate,verbatim,diagbox}
\usepackage[all]{xy}
\usepackage{enumerate}
\usepackage{tikz}
\usepackage{mathrsfs}
\usepackage{graphicx}

\DeclareMathOperator{\Char}{char}

\newcommand{\N}{\mathbb{N}}

\newcommand{\Pan}{{\operatorname{Pan}}}
\newcommand{\CAE}{{\operatorname{CE}}}
\newcommand{\reg}{{\operatorname{reg}}}

\newcommand{\kk}{\Bbbk}

\newtheorem{pro}{Proposition}[section]
\newtheorem{Lem}[pro]{Lemma}
\newtheorem{Cor}[pro]{Corollary} 
\newtheorem{Theo}[pro]{Theorem}

\theoremstyle{definition}
\newtheorem{Defi}[pro]{Definition}
\newtheorem{Exem}[pro]{Example}

\newtheorem{Conj}[pro]{Conjecture}

\newtheorem{Ques}[pro]{Question}

\numberwithin{equation}{section}
\begin{document}

\title[WLP of algebras associated to paths and cycles]{The weak Lefschetz property of artinian algebras associated to paths and cycles}

\author[H.D. Nguyen]{Hop D. Nguyen}
\address{Institute of Mathematics, VAST, 18 Hoang Quoc Viet, Cau Giay, 10307 Hanoi, Vietnam}
\email{ngdhop@gmail.com}

\author[Q.H. Tran]{Quang Hoa Tran}
\address{University of Education, Hue University,  34 Le Loi St., Hue City, Vietnam.}
\email{tranquanghoa@hueuni.edu.vn}

\dedicatory{Dedicated to Professor Ngo Viet Trung on the occasion of his 70th birthday}

\begin{abstract}
Given a base field $\kk$ of characteristic zero, for each graph $G$, we associate the artinian algebra $A(G)$ defined by the edge ideal of $G$ and the squares of the variables. We study the weak Lefschetz property of $A(G)$. We classify some classes of graphs with relatively few edges, including paths and cycles, such that its associated artinian ring has the weak Lefschetz property.
\end{abstract}

\makeatletter
\@namedef{subjclassname@2020}{%
	\textup{2020} Mathematics Subject Classification}
\makeatother


\maketitle
%

\section{Introduction}
A graded artinian algebra $A=[A]_{0} \oplus [A]_{1} \oplus \cdots \oplus [A]_{D}$ over a field $\kk$ has the weak Lefschetz property (WLP for short) if there exists a linear form $\ell \in [A]_{1}$ such that each multiplication maps $\cdot \ell:[A]_{i} \rightarrow [A]_{i+1}$ have maximal rank for all $i$, while $A$ has the strong Lefschetz property (SLP for short) if there exists a linear form $\ell$ such that each multiplication map $\cdot \ell^j: [A]_{i}\longrightarrow [A]_{i+j}$ has maximal rank for all $i$ and all $j$. The study of Lefschetz properties of graded algebras has connections to several areas of mathematics. Many authors have studied the problem from many different points of view, applying tools from representation theory, algebraic topology, differential geometry, commutative algebra, among others (see, for instance, \cite{BK2007, HSS2011,   MMO2013, MMR03,MMR17, MMN2011, MMN2012,MN2013,  MiroRoig2016,  MRT19, MRT2019, MRT2020, Stanley1980}). Even the characteristic of $\kk$ plays an interesting role in the study of the Lefschetz properties; see, for example, \cite{BK2011,CookII2012,CN2011,KV2011,LF2010,MMN2011}. 

The case of artinian $\kk$-algebras defined by monomial ideals, while being rather accessible, is far from simple and the literature concerning their Lefschetz properties is quite extensive; see, for instance, \cite{AB2020,AN2018,DaoNair2022, GLN2020, MM2016,MMO2013,MMN2011,MNS2020}  and the references therein. In this work, we focus on a special class of artinian algebras defined by quadratic monomials, given as follows. Let $G=(V,E)$ be a simple graph on the vertex set $V=\{1,2,\ldots,n\}$ and $R=\kk[x_1,\ldots,x_n]$ be the standard graded polynomial ring over $\kk$. The \emph{edge ideal} of $G$ is given by
$$
I(G)=(x_ix_j\mid \{i,j\}\in E)\subset R.
$$
Then, we say that 
$$
A(G)=\frac{R}{(x_1^2,\ldots,x_n^2)+I(G)}
$$
is the {\it artinian algebra associated to $G$}. We are interested in the following question.
\begin{Ques}
\label{quest_WLPofAG}
For which graphs $G$ does $A(G)$ have the WLP or the SLP? If $A(G)$ does not have the WLP or the SLP, in which degrees do the multiplication maps fail to have maximal rank?
\end{Ques}
The algebra $A(G)$ has been studied in \cite{QHT2020} where the second author classifies the WLP/SLP for some special classes of graphs including the complete graphs, the star graphs, the Barbell graphs, and the wheel graphs. Note that artinian algebras defined by quadratic monomial relations were considered in previous work by Michałek--Mir\'o-Roig \cite{MM2016}, and Migliore--Nagel--Schenck \cite{MNS2020}. These rings can be regarded as special cases of a more general construction due to Dao and Nair \cite{DaoNair2022} where they associate to each simplicial complex $\Delta$ on $n$ vertices the ring
\[
A(\Delta)=\frac{R}{\left( x_{1}^{2}, \ldots, x_{n}^{2}\right) + I_{\Delta}},
\]
in which $I_{\Delta}$ is the Stanley-Reisner ideal of $\Delta$. A recent work due to Cooper et al. \cite{Cooperetal2023}  also investigates the WLP of $A(G)$ where the focus was on whiskered graphs.

Our main goal in this note is to classify some important classes of graphs $G$ where $A(G)$ has the weak Lefschetz property, such as paths, cycles, and certain tadpole graphs. More precisely, denote  by $P_n, C_n, \Pan_n$ the paths, cycles, pan graphs (namely a cycle together with a pendant attached to one vertex), respectively. Our main results are the following.
\begin{Theo}[Theorems~\ref{thm_WLP_paths}, \ref{thm_WLP_cycles} and~\ref{thm_WLP_pan}]
Assume that $\Char(\kk)=0$.
\begin{enumerate}
\item[\rm (1)] For an integer $n\ge 1$, the ring $A(P_n)$ has the WLP if and only if $n\in \{1,2,\ldots,7,9,10,13\}.$ 
\item[\rm (2)] For an integer $n\ge 3$, the ring $A(C_n)$ has the WLP if and only if $n\in \{3,4,\ldots,11,13,14,17\}.$  
\item[\rm (3)] For an integer $n\ge 3$, the ring $A(\Pan_n)$ has the WLP if and only if  $n\in \{3,4,\ldots,10,12,13,16\}.$ 
\end{enumerate}
\end{Theo}
The proof combines Macaulay2 \cite{Macaulay2} computations with inductive arguments based on the unimodality of the independence polynomials of the relevant graphs. We  hope that our main results will inspire further research on Question \ref{quest_WLPofAG}.

Our paper is structured as follows. In the next section we recall relevant terminology and results on artinian algebras, Lefschetz properties, and graph theory. In Section~\ref{Section3}, we investigate the unimodality and the mode of the independence polynomials of familiar graphs, such as paths, cycles and pan graphs. These results are useful to study the WLP of artinian algebras associated to these graphs. In Section~\ref{Section4}, we prove the our main theorems (see Theorems~\ref{thm_WLP_paths},~\ref{thm_WLP_cycles} and~\ref{thm_WLP_pan}) on the WLP of the artinian algebras associated to paths, cycles and the pan graphs.

\section{Preliminaries}
In this section we recall some standard terminology and notations from commutative algebra and combinatorial commutative algebra, as well as some results needed later on. For a general introduction to artinian rings and the weak and strong Lefschetz properties we refer the readers to \cite{HMMNWW2013} and  \cite{MN2013}.
\subsection{ The weak Lefschetz property}

In this paper we consider artinian algebras defined by monomial ideals, and in this case it suffices to chose the Lefschetz element to be the sum of the variables.
\begin{pro}{\rm \cite{MMN2011, Nick2018}}\label{Proposition2.5}
Let $I\subset R=\kk[x_1,\ldots,x_n]$ be an artinian monomial ideal. Then $A=R/I$ has the WLP if and only if $\ell=x_1+x_2+\cdots+x_n$ is a Lefschetz element for $A$.
\end{pro}

A necessary condition for the WLP and SLP of an artinian algebra $A$ is the unimodality of the Hilbert series of $A$. 
\begin{Defi}
Let $A=\oplus_{j\geq 0} [A]_j$ be a standard graded $\kk$-algebra. The {\it Hilbert series} of $A$ is the power series $\sum\dim_\kk [A]_i t^i$ and is denoted by $HS(A,t)$. The {\it Hilbert function} of $A$ is the function  $h_A: \N\longrightarrow \N$ defined by $h_A(j)=\dim_\kk [A]_j$.\\
If $A$ is an artinian graded algebra, then $[A]_i=0$ for $i\gg 0.$  Denote
$$D=\max\{i\mid [A]_i\neq 0\},$$
the \emph{socle degree} of $A$. In this case, the Hilbert series of $A$ is a polynomial
$$HS(A,t)=1+h_1t+\cdots+ h_Dt^D,$$
where $h_i=\dim_\Bbbk [A]_i>0$. By definition, the degree of the Hilbert series for an artinian graded algebra $A$ is equal to its socle degree $D.$ Since $A$ is artinian and non-zero, this number also agrees with the {\it Castelnuovo-Mumford regularity} of $A$, i.e,
$$\reg (A)= D =\deg (HS(A,t)).$$
The algebra $A$ is called \emph{level} if its socle is concentrated in one degree. 
\end{Defi}

\begin{Defi}
A polynomial $\sum_{k=0}^na_kt^k \in \mathbb{R}[t]$ with non-negative coefficients is called {\it unimodal} if there is some $m$, such that
$$a_0\leq a_1\leq \cdots \leq a_{m-1}\leq a_m \geq a_{m+1}\geq \cdots \geq a_n.$$ 
Set $a_{-1}=0$. The {\it mode} of the unimodal polynomial $\sum_{k=0}^na_kt^k$ is defined to be the unique integer $i$ between $0$ and $n$ such that
$$
a_{i-1}< a_i \geq a_{i+1}\geq \cdots \geq a_n.
$$
\end{Defi}

\begin{pro}{\rm \cite[Proposition~3.2]{HMMNWW2013}}\label{Proposition2.8}
If $A$ has the WLP or SLP then the Hilbert series of $A$ is unimodal.
\end{pro}

Finally, to study the failure of the WLP of tensor products of $\kk$-algebras, the following simple lemma turns out to be quite useful.

\begin{Lem}{\rm\cite[Lemma 7.8]{BMMNZ12}}\label{lem_tensor}
Let $A=A'\otimes_\kk A''$ be a tensor product of two graded artinian $\kk$-algebras $A'$ and $A''$. Let $\ell'\in A'$ and $\ell''\in A''$ be linear elements, and set $\ell=\ell'+\ell''=\ell'\otimes 1 + 1\otimes \ell''\in A$. Then
\begin{enumerate}
\item [\rm (a)] If the multiplication maps $\cdot\ell': [A']_{i}\longrightarrow [A']_{i+1}$ and $\cdot\ell'': [A'']_{j}\longrightarrow [A'']_{j+1}$ are both not surjective, then neither is the map
$$
\cdot\ell: [A]_{i+j+1}\longrightarrow [A]_{i+j+2}.
$$
\item [\rm (b)] If the multiplication maps $\cdot\ell': [A']_{i}\longrightarrow [A']_{i+1}$ and $\cdot\ell'': [A'']_{j}\longrightarrow [A'']_{j+1}$ are both not injective, then neither is the map
$$
\cdot\ell: [A]_{i+j}\longrightarrow [A]_{i+j+1}.
$$

\end{enumerate}
\end{Lem}

\subsection{Graph theory}
From now on, by a graph we mean a simple graph $G=(V,E)$ with the vertex set $V=V(G)$ and the edge set $E=E(G)$. 
We start by recalling some basic definitions. 
\begin{Defi} 
The {\it disjoint union} of the graphs $G_1$ and $G_2$ is a graph $G=G_1\cup G_2$ having as vertex set the disjoint union of $V(G_1)$ and $V(G_2)$, and as edge set the disjoint union of $E(G_1)$ and $E(G_2)$. In particular, $\cup_{n}G$ denotes the disjoint union of $n>1$ copies of the graph $G$.
\end{Defi}
\begin{Defi} Let $G=(V,E)$ be a graph.
\begin{enumerate}[\quad \rm (i)]
\item  A subset $X$ of $V$ is called an {\it independent set} of $G$ if for any $i,j\in X,\ \{i,j\}\notin E$, i.e., the vertices in $X$ are pairwise non-adjacent. If an independent set $X$ has $k$ elements, then we say that $X$ is an {\it independent set of size $k$} or a $k$-independent set of $G$.
\item  An independent set $X$ is called {\it maximal} if for every vertices $v\in V\setminus X$, $X\cup \{v\}$ is not an independent set of $G$.
\item The {\it independence number} of a graph $G$ is the largest cardinality of an independent set of $G$. We denote this value by $\alpha(G)$.
\item  A graph $G$ is said to be  {\it well-covered} if every maximal independent set of $G$ has the same size $\alpha(G).$
\end{enumerate}
\end{Defi}
\begin{Defi}
The {\it independence polynomial} of a graph $G$ is a polynomial in one variable $t$ whose coefficient of $t^k$ is given by the number of independent sets of size $k$ of $G$. We denote this polynomial by $I(G;t)$, i.e.,
$$I(G;t)=\sum_{k=0}^{\alpha(G)}s_k(G)t^k,$$
where $s_k(G)$ is the number of independent sets of size $k$ in $G$. Note that $s_0(G)=1$ since $\emptyset$ is an independent set of any graph $G$.
\end{Defi}
The independence polynomial of a graph was defined by Gutman and Harary in \cite{GH83} as a generalization of the matching polynomial of a graph. For a vertex $v\in V$, define $N(v) = \{w \mid  w \in V \;\text{and}\; \{v,w\}\in E\}$ and
$N[v] = N(v) \cup \{v\}$. The following equalities are very useful for the calculation of the independence polynomial for various families of graphs (see, for instance, \cite{GH83,HL94}).
\begin{pro}\label{FormulaforIndependence}
Let $G_1,G_2, G$ be the graphs. Assume that $G=(V,E), w\in V$ and  $e=\{u,v\}\in E$. Then the following equalities hold:
\begin{enumerate}
\item [\rm (i)] $I(G;t)=I(G\setminus w;t)+t\cdot I(G\setminus N[w];t)$;
\item [\rm (ii)] $I(G;t)=I(G\setminus e;t)-t^2\cdot I(G\setminus (N(u)\cup N(v));t)$;
\item [\rm (iii)] $I(G_1\cup G_2;t)=I(G_1;t)I(G_2;t)$.
\end{enumerate}
 \end{pro}

\subsection{Artinian algebras associated to graphs}
Let $G=(V,E)$ be a graph, with the set of vertices $V=\{1,2,\ldots,n\}$. Let $R=\kk[x_1,\ldots,x_n]$ be a standard graded polynomial ring over $\kk$. The edge ideal of $G$ is the ideal
$$I(G)=(x_ix_j\mid \{i,j\}\in E)\subset R.$$
Then, we say that 
$$A(G)=\frac{R}{(x_1^2,\ldots,x_n^2)+I(G)}$$
is the {\it artinian algebra associated to $G$}. The algebra $A(G)$ contains significant combinatorial information about $G$, as witnessed by
\begin{pro}
The Hilbert series of $A(G)$ is equal to the independence polynomial of $G$, i.e.
$$HS(A(G);t)=I(G;t)=\sum_{k=0}^{\alpha(G)}s_k(G)t^k.$$
As a consequence, the Castelnuovo--Mumford regularity of $A(G)$ is $\reg(A(G))=\alpha(G)$ and $A(G)$ is level if and only $G$ is well-covered.
\end{pro}
Therefore, the WLP/SLP of $A(G)$ has strong consequences on the unimodality of the independence polynomial of $G$. Indeed, if $I(G;t)$ is not unimodal, then $A(G)$ fails the WLP by Proposition~\ref{Proposition2.8}. Thus, to study  the WLP/ SLP of $A(G)$, it is enough to consider the graphs whose independence polynomials are unimodal. Concerning the unimodality of the independence polynomial of graphs, we have the following famous conjecture.
\begin{Conj}\cite{AEMS87}\label{unimodalityconjecture}
If $G$ is a tree or forest, then the independence polynomial of $G$ is unimodal.
\end{Conj}

To our best knowledge, until now, the largest class of graphs for which the independence polynomial is known to be unimodal is the class of claw-free graphs \cite{Hamidoune90}. Recall that a graph is said to be \emph{claw-free} if it does not admit  the complete bipartite graph $K_{1,3}$ as an induced subgraph.  Conjecture~\ref{unimodalityconjecture} remains widely open. The following example due to Bhattacharyya and Kahn \cite{BK2013} shows that one cannot expect the statement of Conjecture~\ref{unimodalityconjecture} to be true for bipartite graphs.
\begin{Exem}
Given positive integers $m$ and $n>m$, let $G=(V,E)$ with $V=V_1\cup V_2\cup V_3$, where $V_1,V_2,V_3$ are disjoint; $|V_1|=n-m$ and $|V_2|=|V_3|=m;$ $E$ consists of the edges of the complete bipartite graph with the bipartition $V_1 \cup V_2$ and a perfect matching between $V_2$ and $V_3$. Then $G$ is a bipartite graph and for every $i\geq 0,\ s_i(G)
=(2^i-1)\binom{m}{i}+\binom{n}{i}.$ Therefore, for $m\geq 95$ and $n=\lfloor m\log_2(3)\rfloor$, $I(G;t)$ is not unimodal. As a consequence, $A(G)$ fails the WLP.
\end{Exem}
It is known that the Lefschetz properties depend strongly on the characteristic of the field.

\begin{Exem}
An empty graph is simply a graph with no edges. We denote the empty graph on $n$ vertices by $E_n$. Then
$$A(E_n)=R/(x_1^2,\ldots,x_n^2)\quad \text{and}\; I(E_n;t)=(1+t)^n.$$
A well-known result of Stanley  on complete intersections says that $A(E_n)$ has the SLP if $\Char(\kk)=0$ or $\Char(\kk)> n$ \cite{Stanley1980,HT2023}. This result does not hold if  $\Char(\kk)\le n$, even for the WLP. Indeed, in the case where $\Char(\kk)=2$, it was known that $A(E_n)$ has the WLP if and only if $n=3$ \cite{BK2011,KV2011}. In  \cite{KV2011} Kustin and Vraciu showed that if $ n\ge 5$ and $\Char(\kk)=p$ is odd, then $A(E_n)$ has the WLP if and only if $p\ge \lfloor\frac{n+3}{2}\rfloor$.
\end{Exem}
The complete graph on $n$ vertices, denoted by $K_n$, is the graph where each vertex is adjacent to every other. It follows that
$$A(K_n)=R/(x_1,\ldots,x_n)^2\quad \text{and}\; I(K_n;t)=1+nt.$$
It is easy to see that $A(K_n)$ has the SLP for any field $\kk$. Concerning disjoint unions of complete graphs, we have the following.
\begin{pro} {\rm \cite[Theorem 4.8]{MNS2020}}
Let $\Char(\kk)\neq 2$ and $A(G)$ be the artinian algebra associated to $G=\cup_{i=1}^r K_{n_i}$. Assume $n_1\geq n_2\geq \cdots\geq n_r\geq 1$. Then $A(G)$ has the WLP if and only if one of the following holds:
\begin{enumerate}
\item [\rm (1)] $n_2=\cdots =n_r=1$, i.e. $G$ is the disjoint union of a complete graph $K_{n_1}$ and an empty graph on $r-1$ vertices.
\item [\rm (2)] $n_3=\cdots =n_r=1$ and $r$ is odd.
\end{enumerate}
In particular for every $n\geq 2$,  if $G$ is the disjoint union of $n$ complete graphs, none of which is a singleton, then $A(G)$ does not have the WLP.
\end{pro}

\section{Independence polynomial of some graphs}\label{Section3}
In this section, we provide some results on the independence polynomial of some familiar graphs, namely paths, cycles, and pan graphs. These results will be useful to prove our main theorems in the next section.
\subsection{Paths}
Let $P_n$ be the path on $n$ vertices ($n\ge 1$).
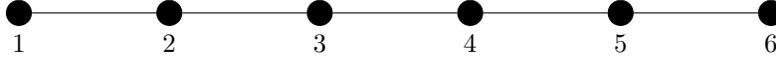
\begin{figure}[htp]
	\centering
	\begin{tikzpicture}
	[
every edge/.style = {draw=black},
vrtx/.style args = {#1/#2}{%
	circle, draw, fill=black,
	minimum size=1mm, label=#1:#2}
]
		\node (n1) [vrtx=below/1] at (-5,0) {};
		\node (n2) [vrtx=below/2]at (-3,0)  {};
		\node (n3) [vrtx=below/3]at (-1,0)  {};
		\node (n4) [vrtx=below/4]at (1,0) {};
		\node (n5) [vrtx=below/5]at (3,0)  {};
		\node (n6) [vrtx=below/6]at (5,0)  {};
		\foreach \from/\to in {n1/n2,n2/n3,n3/n4,n4/n5,n5/n6}
		\draw (\from) -- (\to);	
	\end{tikzpicture}
	\caption{The path $P_6$}
\end{figure}

\begin{pro}\label{IndependencePolynomialofPath}
	The independence polynomial of $P_n$ is
	\[
	I(P_n;t)=\sum_{i=0}^{\lfloor \frac{n+1}{2}\rfloor} \binom{n+1-i}{i} t^i.
	\]
	Moreover, for every $n\ge 1$, $I(P_n;t)$ is unimodal, with the mode 
	$$
	\lambda_n=\left\lceil \dfrac{5n+2-\sqrt{5n^2+20n+24}}{10} \right \rceil.
	$$ 
\end{pro}
\begin{proof}
Hopkins and Staton \cite{GL84} showed that
$$I(P_n; t) = F_{n+1}(t),$$
where $F_n(t),\ n\geq 0$, are the so-called Fibonacci polynomials, which are defined recursively by
$$F_0(t) = 1; F_1(t)= 1; F_n(t) = F_{n-1}(t) + tF_{n-2}(t).$$
Based on this recurrence, one can deduce that
\[
I(P_n;t)=\sum_{i=0}^{\lfloor \frac{n+1}{2}\rfloor} \binom{n+1-i}{i} t^i.
\]
The unimodality of the independence polynomial of $P_n$ is implied from the fact that the independence polynomial of a claw-free graph  is unimodal \cite{Hamidoune90}. Now we determine the mode of $I(P_n;t)$.
Let $i$ be an integer such that $0\le i\le \lfloor \frac{n+1}{2}\rfloor$ and  
\[
\binom{n+1-i}{i}\ge \binom{n-i}{i+1}. 
\]
This is clearly true if $i\ge n/2$. If $i< n/2$, we have
\begin{align*}
&\binom{n+1-i}{i}\ge \binom{n-i}{i+1}\\
&\Leftrightarrow\frac{n+1-i}{(n-2i)(n-2i+1)}\ge \frac{1}{i+1}\\
&\Leftrightarrow 5i^2-(5n+2)i+n^2-1\le 0\\
&\Leftrightarrow \frac{5n+2-\sqrt{5n^2+20n+24}}{10}\le i\le \frac{5n+2+\sqrt{5n^2+20n+24}}{10}.
\end{align*}
As the inequality on the right holds for any $i\le \lfloor \frac{n+1}{2}\rfloor$, we have
\[
\binom{n+1-i}{i}\ge \binom{n-i}{i+1} \Leftrightarrow  i\ge \frac{5n+2-\sqrt{5n^2+20n+24}}{10}.
\]
This means that the mode of $I(P_n;t)$ is equal to $\lambda_n=\lceil\frac{5n+2-\sqrt{5n^2+20n+24}}{10}\rceil.$
\end{proof}
We summarize the important properties of the mode of $I(P_n;t)$.
\begin{Lem}\label{lem_compare_lambda}
For any $n\geq 1$, one has the following.
\begin{enumerate}
\item [\rm (i)]  $\lambda_{n+1}\geq \lambda_n;$
\item [\rm (ii)] $\lambda_{n+3}-1\leq \lambda_n\leq \lambda_{n+4}-1;$
\item [\rm (iii)]  $\lambda_{n+11}\geq \lambda_n+3.$
\end{enumerate} 
\end{Lem}
\begin{proof}
Set $\alpha_n=\frac{5n+2-\sqrt{5n^2+20n+24}}{10}.$ A straightforward computation shows that 
\begin{align*}
\alpha_{n+1}\geq \alpha_n;\; 
\alpha_{n+3}-1\leq \alpha_n\leq \alpha_{n+4}-1 \; \text{and}\;\alpha_{n+11}\geq \alpha_n+3.
\end{align*}
The lemma follows from basic properties of the ceiling function.
\end{proof}

Table \ref{tab_indpoly} provides information about the initial values of the mode of  the independence polynomial $I(G;t)$ for the classes of graphs considered in this paper, by using {\tt Macaulay2} \cite{M2codes}. A dash indicates an undefined value. 
\begin{table}[ht!]
\caption{Graphs and modes of their independence polynomials}
\label{tab_indpoly}
\begin{tabular}{| c  | c  | c | c | c | c | c | c | c |  c | c |  c  |  c  | c  | c | c |  }
\hline
$G$ & \diagbox{\text{mode}\\ \text{of $I(G;t)$}}{\\ \\$n$}  & 1 & 2 & 3 & 4 & 5 & 6 & 7 & 8 & 9 & 10 & 11 & 12 & 13 \\
\hline
$P_n$ & $\lambda_n$        &     0  & 1  & 1  & 1  & 2  & 2 & 2 & 2 & 3  & 3 & 3 & 4 & 4\\
\hline
$C_n$ & $\rho_n$          &     -  & -  & 1  & 1  & 1  & 2 & 2 & 2 & 3 & 3 &  3 & 3 & 4 \\
\hline
$\CAE_n$ & $\chi_n$          &     -   & -  & -  & 1  & 1  & 2  & 2  & 2 & 2 & 3 & 3  & 3 & 4 \\
\hline
$\Pan_n$ & $\zeta_n$              & -   & -  &  1 & 1  & 2  & 2  & 2 & 3  & 3 & 3 & 3  & 4 & 4\\
\hline
\end{tabular}
\end{table}

\subsection{Cycles}
Let $C_n$ be the cycle on $n$ vertices ($n\geq 3$). 
\begin{figure}[!h]
	\centering
	\begin{tikzpicture}
	[
	every edge/.style = {draw=black,very thick},
	vrtx/.style args = {#1/#2}{%
		circle, draw, thick, fill=black,
		minimum size=1mm, label=#1:#2}
	]
		\node (n1) [vrtx=above/2] at (-1,0) {};
		\node (n2) [vrtx=above/1] at (1,0)  {};
		\node (n3) [vrtx=right/6]at (2,-1.5)  {};
		\node (n4) [vrtx=below/5]at (1,-3) {};
		\node (n5) [vrtx=below/4]at (-1,-3)  {};
		\node (n6) [vrtx=left/3]at (-2,-1.5)  {};
		\foreach \from/\to in {n1/n2,n2/n3,n3/n4,n4/n5,n5/n6,n6/n1}
		\draw (\from) -- (\to);	
	\end{tikzpicture}
	\caption{The cycle $C_6$}
\end{figure}
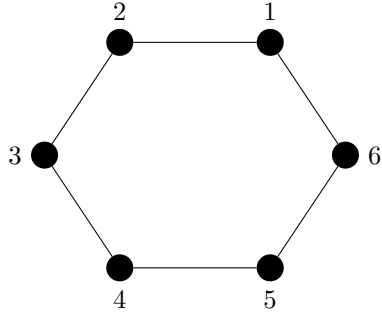
\begin{pro}\label{Prop_independenceofCn}
The independence polynomial of $C_n$ is
	\begin{align*}
		I(C_n;t)&=I(P_{n-1};t)+tI(P_{n-3};t)\\
		&=1+\sum_{i=1}^{\lfloor \frac{n}{2}\rfloor} \frac{n}{i}\binom{n-i-1}{i-1} t^i.
	\end{align*}
	Moreover, $I(C_n;t)$ is unimodal, with the mode $\rho_n=\lceil\frac{5n-4-\sqrt{5n^2-4}}{10}\rceil$ for all $n\ge 3$.
\end{pro}
\begin{proof}
Hopkins and Staton \cite{GL84} showed that
$$I(C_n; t) =1+\sum_{i=1}^{\lfloor \frac{n}{2}\rfloor} \frac{n}{i}\binom{n-i-1}{i-1} t^i.$$
The unimodality of the independence polynomial of $C_n$ is implied from the fact that the independence polynomial of a claw-free graph is unimodal \cite{Hamidoune90}. Arguing as in the proof of Proposition~\ref{IndependencePolynomialofPath}, solving for $1\le i \le \lceil \frac{n}{2} \rceil-1$ from 
\[
\frac{n}{i}\binom{n-i-1}{i-1} \ge \frac{n}{i+1}\binom{n-i-2}{i}
\]
we get
\begin{align*}
 (i+1)(n-i-1)\ge (n-2i-1)(n-2i).
\end{align*}
Equivalently
$5i^2-i(5n-4)+n^2-2n+1\le 0$. Thus
\[
\frac{5n-4-\sqrt{5n^2-4}}{10} \le i\le \frac{5n-4+\sqrt{5n^2-4}}{10}.
\]
This implies that the mode of $I(C_n;t)$ is equal to $\rho_n=\left\lceil\frac{5n-4-\sqrt{5n^2-4}}{10}\right\rceil$, as desired.
\end{proof}
\begin{Lem}
	\label{lem_compare_rho_lambda}
	For all $n\ge 5$, there are inequalities
	$\lambda_{n-1}\leq \rho_n\leq \lambda_{n-4}+1 \leq \lambda_n$.
\end{Lem}
\begin{proof}
	By Lemma \ref{lem_compare_lambda}, $\lambda_{n-4}+1 \leq \lambda_n$, hence it suffices to show that
	\[
	\lambda_{n-1}\leq \rho_n\leq \lambda_{n-4}+1.
	\]
	For the inequality on the left, we have to show that
	\begin{align*}
		&\dfrac{5(n-1)+2-\sqrt{5(n-1)^2+20(n-1)+24}}{10} \le \dfrac{5n-4-\sqrt{5n^2-4}}{10}\\
		&\Leftrightarrow 5n-3-\sqrt{5n^2+10n+9} \le 5n-4-\sqrt{5n^2-4}\\
		& \Leftrightarrow \sqrt{5n^2-4}+1 \le \sqrt{5n^2+10n+9} \\
		& \Leftrightarrow 5n^2-3+2\sqrt{5n^2-4} \le 5n^2+10n+9\\
		& \Leftrightarrow \sqrt{5n^2-4} \le 5n+6 \Leftrightarrow (5n+6)^2-(5n^2-4) \ge 0\\
		& \Leftrightarrow 20n^2+60n+40 \ge 0,
	\end{align*}
	which is clear.
	
		For the inequality on the right, we have to show that
	\begin{align*}
		&\dfrac{5n-4-\sqrt{5n^2-4}}{10} \le \dfrac{5(n-4)+2-\sqrt{5(n-4)^2+20(n-4)+24}}{10}+1 \\
		&\Leftrightarrow 5n-4-\sqrt{5n^2-4} \le 5n-8-\sqrt{5n^2-20n+24}\\
		& \Leftrightarrow \sqrt{5n^2-20n+24}+4 \le \sqrt{5n^2-4} \\
		& \Leftrightarrow 5n^2-20n+24+16+8\sqrt{5n^2-20n+24} \le 5n^2-4 \quad (\text{by squaring})\\
		& \Leftrightarrow 8\sqrt{5n^2-20n+24} \le 20n-44 \\
		& \Leftrightarrow 2\sqrt{5n^2-20n+24} \le 5n-11\\
		& \Leftrightarrow 4(5n^2-20n+24) \le (5n-11)^2 \\
		& \Leftrightarrow 5n^2-30n+25 \ge 0  \Leftrightarrow 5(n-1)(n-5)\ge 0
	\end{align*}
	which is true for all $n\ge 5$. The proof is completed.
\end{proof}

\subsection{Pans}

The $n$-pan graph is the graph obtained by joining a cycle graph $C_n$ to a singleton graph $K_1$ with a bridge.  We denote this graph by $\Pan_n$. 

Our goal is to show the independence polynomial of $\Pan_n$ is unimodal. For this, we consider a family of graphs formed by adding an edge $\{n-2,n\}$ to the cycles $C_n$ ($n\geq 4$). We denote this graph by $\CAE_n$. 

\begin{figure}[!h]
	\begin{tabular}{cc}
\begin{tikzpicture}[
	every edge/.style = {draw=black,very thick},
	vrtx/.style args = {#1/#2}{%
		circle, draw, thick, fill=black,
		minimum size=1mm, label=#1:#2}
		]
		\node (n1) [vrtx=above/2]  at (-1,0) {};
		\node (n2) [vrtx=above/1]at (1,0)  {};
		\node (n3) [vrtx=below/6]at (2,-1.5)  {};
		\node (n4) [vrtx=below/5]at (1,-3) {};
		\node (n5) [vrtx=below/4]at (-1,-3)  {};
		\node (n6) [vrtx=left/3]at (-2,-1.5)  {};
		\node (n7) [vrtx=below/7]at (4,-1.5)  {};
		\foreach \from/\to in {n1/n2,n2/n3,n3/n4,n4/n5,n5/n6,n6/n1, n3/n7}		
		\draw (\from) -- (\to);	
	\end{tikzpicture} &
	\begin{tikzpicture}[
		every edge/.style = {draw=black,very thick},
		vrtx/.style args = {#1/#2}{%
			circle, draw, thick, fill=black,
			minimum size=1mm, label=#1:#2}
		]
		\node (n1) [vrtx=above/2]  at (-1,0) {};
		\node (n2) [vrtx=above/1]at (1,0)  {};
		\node (n3) [vrtx=right/6]at (2,-1.5)  {};
		\node (n4) [vrtx=below/5]at (1,-3) {};
		\node (n5) [vrtx=below/4]at (-1,-3)  {};
		\node (n6) [vrtx=left/3]at (-2,-1.5)  {};
		\foreach \from/\to in {n1/n2,n2/n3,n3/n4,n4/n5,n5/n6,n6/n1, n3/n5}
		\draw (\from) -- (\to);	
	\end{tikzpicture}
\end{tabular}
	\caption{$\Pan_6$ and $\CAE_6$}
\end{figure}

Note that $\CAE_n$ is a claw-free graph, and hence its independence polynomial is unimodal \cite{Hamidoune90}. 
\begin{Lem}\label{lem_compare_lambda_chi}
The independence polynomial of $\CAE_n$ is
\begin{align*}
	I(\CAE_n;t)&=\sum_{i=0}^{\alpha(\CAE_n)} s_i(\CAE_n)t^i\\
	&=I(P_{n-1};t)+tI(P_{n-4};t)\\
	&=\sum_{i=0}^{\lfloor\frac{n}{2} \rfloor}\bigg[\binom{n-i}{i}+\binom{n-i-2}{i-1}\bigg]t^i.
\end{align*}
Let $\chi_n$ be the mode of $I(\CAE_n;t)$ and $\lambda_n$ be the mode of $I(P_n;t)$ as in Proposition \ref{IndependencePolynomialofPath}. For any $n\geq 5$, one has
	$\lambda_{n-1}\leq \chi_n\leq \lambda_{n-4}+1.$
\end{Lem}
\begin{proof}
The first assertion follows from applying  Proposition~\ref{FormulaforIndependence}(i) for the vertex numbered $n$.

Let $1\le i\leq \lambda_{n-1}$. We need to show that 
	$$s_{i-1}(\CAE_n)<s_i(\CAE_n),$$
namely, 
\[
\binom{n-i+1}{i-1}+\binom{n-i-1}{i-2}< \binom{n-i}{i}+\binom{n-i-2}{i-1}.
\]
This is clear for $i=1$, so we assume that $i\ge 2$.

Since $i\leq \lambda_{n-1}$, $\binom{n-i+1}{i-1}\le\binom{n-i}{i}$. It suffices to show that
	 \begin{align*}
		&\binom{n-i-1}{i-2}< \binom{n-i-2}{i-1}\\
		&\Leftrightarrow (n-i-1)(i-1)< (n-2i)(n-2i+1)\\
		&\Leftrightarrow 5i^2-(5n+2)i+n^2+2n-1> 0\\
		&\Leftrightarrow i< \dfrac{5n+2-\sqrt{5n^2-20n+24}}{10}\ \text{or}\ i> \dfrac{5n+2+\sqrt{5n^2-20n+24}}{10}.
	\end{align*}
	
	As $i\leq \lambda_{n-1}$, it is enough to show that 
	\begin{align*}
		&\dfrac{5(n-1)+2-\sqrt{5(n-1)^2+20(n-1)+24}}{10}< \dfrac{5n+2-\sqrt{5n^2-20n+24}}{10} -1\\
		&\Leftrightarrow 5n-3-\sqrt{5n^2+10n+9} < 5n-8-\sqrt{5n^2-20n+24}\\
		&\Leftrightarrow 5+\sqrt{5n^2-20n+24} < \sqrt{5n^2+10n+9}\\
		&\Leftrightarrow \sqrt{5n^2-20n+24}< 3n-4 \quad(\text{after squaring and simplifying})\\
		&\Leftrightarrow n^2-n-2> 0\\
		&\Leftrightarrow (n+1)(n-2)> 0,
	\end{align*}
	which is clear for any $n\ge 4$. It follows that $\lambda_{n-1}\le \chi_n.$ It remains to show that if $\lfloor\frac{n}{2} \rfloor \ge i\geq \lambda_{n-4}+1$ (note that $\lfloor\frac{n}{2} \rfloor$ is the independence number of $\CAE_n$), then 
$$s_{i}(\CAE_n)\ge s_{i+1}(\CAE_n)\Longleftrightarrow \binom{n-i}{i}+\binom{n-i-2}{i-1}\ge \binom{n-i-1}{i+1}+\binom{n-i-3}{i}.$$
	By Lemma \ref{lem_compare_lambda}, $i\geq \lambda_{n-4}+1\ge \lambda_{n-1}$, so $\binom{n-i}{i}\ge\binom{n-i-1}{i+1}$ thanks to Proposition \ref{IndependencePolynomialofPath}. We have to show that
	\begin{align*}
		\binom{n-i-2}{i-1}&\ge \binom{n-i-3}{i}\\
		&\Leftrightarrow i(n-i-2)\ge (n-2i-2)(n-2i-1)\\
		&\Leftrightarrow 5i^2-(5n-8)i+n^2-3n+2\le 0\\
		&\Leftrightarrow \dfrac{5n-8-\sqrt{5n^2-20n+24}}{10}\le i\le \dfrac{5n-8+\sqrt{5n^2-20n+24}}{10}
	\end{align*}
	Since $\lfloor\frac{n}{2} \rfloor \ge i\geq \lambda_{n-4}+1$, and by simple computations,
	\[
    \left\lfloor\frac{n}{2} \right \rfloor \le \dfrac{5n-8+\sqrt{5n^2-20n+24}}{10} \quad \text{for all $n\ge 5$},
	\]
the inequality on the right of the last chain is always true.	Thus it is enough to prove the inequality on the left, which would be true if
	\begin{align*}
		&\dfrac{5n-8-\sqrt{5n^2-20n+24}}{10}\le \dfrac{5(n-4)+2-\sqrt{5(n-4)^2+20(n-4)+24}}{10}+1\\
		&\Leftrightarrow \dfrac{5n-8-\sqrt{5n^2-20n+24}}{10}\le \dfrac{5n-18 +\sqrt{5n^2-20n+24}}{10}+1,
	\end{align*}
	which is clear. Thus $\chi_n \le \lambda_{n-4}+1.$ The proof is completed.  
\end{proof}

Note that $\Pan_n$ is not a claw-free graph. Hence, we need to show the unimodality of its independence polynomial. We have the following. 

\begin{Lem}\label{lem_compare_lambda_chi-zeta}
	The independence polynomial $I(\Pan_n;t)$ of the $n$-pan graph is unimodal. Let $\zeta_n$ be the mode of $I(\Pan_n;t)$. Then for all $n\ge 5$, there are inequalities 
	\begin{equation}
	\label{eq_bouding_zeta}
	 \chi_{n+1}\leq \zeta_n\leq\rho_n+1\leq \lambda_n+1\leq \chi_{n+1}+1,
	\end{equation}
	where $\lambda_n, \rho_n, \chi_n$ are as in Propositions \ref{IndependencePolynomialofPath}, \ref{Prop_independenceofCn} and Lemma \ref{lem_compare_lambda_chi}.
\end{Lem}

\begin{proof}
By using  Proposition~\ref{FormulaforIndependence}(i) for the vertex of $K_1$, Propositions~\ref{IndependencePolynomialofPath} and~\ref{Prop_independenceofCn}, we have
	\begin{align*}
		I(\Pan_n;t)&=\sum_{i=0}^{\alpha(\Pan_n)} s_i(\Pan_n)t^i\\
		&=I(C_{n};t)+tI(P_{n-1};t)\\
		&=I(P_{n-1};t)+t\big(I(P_{n-3};t)+I(P_{n-1};t)\big)\\
		&=\sum_{i=0}^{\lfloor\frac{n}{2}\rfloor+1}\bigg[\binom{n-i}{i}+\binom{n-i-1}{i-1}+\binom{n-i+1}{i-1}\bigg]t^i.
	\end{align*}
	Therefore, we have
	\begin{align}
		s_i(\Pan_n)&=\binom{n-i}{i}+\binom{n-i-1}{i-1}+\binom{n-i+1}{i-1}  \nonumber\\ 
		            &=\binom{n-i+1}{i}+\binom{n-i-1}{i-1}+\binom{n-i}{i-2} \nonumber\\
		            & \qquad \left(\text{using $\binom{n}{p}=\binom{n-1}{p}+\binom{n-1}{p-1}$}\right ) \nonumber\\
		            &=s_i(\CAE_{n+1})+\binom{n-i}{i-2} \label{eq_siPan} \quad \text{(by Lemma \ref{lem_compare_lambda_chi})}. 
	\end{align}
	We first have the following assertion.

	\noindent \textsc{Claim 1:} $s_{i-1}(\Pan_n)<s_i(\Pan_n)$ for any $i\leq \chi_{n+1}.$\\
	
	\emph{Proof of Claim 1:} For any $1\le i\leq \chi_{n+1}$, $s_{i-1}(\CAE_{n+1})\le s_i(\CAE_{n+1}).$  Therefore by Equation \eqref{eq_siPan}, it suffices to show that 
	\begin{align*}
		&\hspace*{1.2cm}\binom{n-i+1}{i-3}< \binom{n-i}{i-2}\\ 
		&\Longleftrightarrow (n-i+1)(i-2) < (n-2i+3)(n-2i+4) \\
		&\Longleftrightarrow 5i^2-(5n+17)i+n^2+9n+14> 0\\
		&\Longleftrightarrow i< \frac{5n+17-\sqrt{5n^2-10n+9}}{10}\ \text{or}\ i> \frac{5n+17+\sqrt{5n^2-10n+9}}{10}.
	\end{align*}
	By Lemma~\ref{lem_compare_lambda_chi},  
	\begin{align*}
i\leq &\chi_{n+1}\leq \lambda_{n-3}+1 \\
      & < \dfrac{5(n-3)+2-\sqrt{5(n-3)^2+20(n-3)+24}}{10} +2 \\
      &= \dfrac{5n+7-\sqrt{5n^2-10n+9}}{10}.
	\end{align*}
	Consequently, it suffices to show that
	\begin{align*}
		\frac{5n+7-\sqrt{5n^2-10n+9}}{10}\leq \frac{5n+17-\sqrt{5n^2-10n+9}}{10},
	\end{align*}
	which is clear. Thus, we finish the proof of Claim 1.

	Now, by again Proposition~\ref{FormulaforIndependence}, we have
	\begin{align*}
		I(\Pan_n;t)&=\sum_{i=0}^{\alpha(\Pan_n)} s_i(\Pan_n)t^i\\
		&=I(C_{n};t)+tI(P_{n-1};t),
	\end{align*}
	we get $s_i(\Pan_n)=s_i(C_n)+s_{i-1}(P_{n-1}).$ Next we have the following
	
	\noindent \textsc{Claim 2:} $s_{i}(\Pan_n)\ge s_{i+1}(\Pan_n)$ for any $i\ge \rho_{n}+1.$\\
	
	\emph{Proof of Claim 2:} 	
	Since $i\geq \rho_{n}+1$ and $n\ge 5$, $i-1\geq \rho_n\geq \lambda_{n-1}$ by Lemma~\ref{lem_compare_rho_lambda}. It follows that $s_i(C_n)\geq s_{i+1}(C_n)$ and $s_{i-1}(P_{n-1})\geq s_i(P_{n-1}).$ Thus 
	$s_i(\Pan_n)\geq s_{i+1}(\Pan_n)$, as desired.
	
	By Lemmas~\ref{lem_compare_rho_lambda} and~\ref{lem_compare_lambda_chi}, $\rho_n\leq \lambda_n\leq \chi_{n+1},$ which yields the last two inequalities in \eqref{eq_bouding_zeta}. Moreover, it follows from Claims 1 and 2 that  $s_{i-1}(\Pan_n)< s_{i}(\Pan_n)$ for any $i\le \chi_{n+1}$ and $s_i(\Pan_n)\geq s_{i+1}(\Pan_n)$ for any $i\ge \chi_{n+1}+1$. Thus, the independence polynomial $I(\Pan_n;t)$ of the $n$-pan graph is unimodal. Moreover,  $\chi_{n+1}\leq \zeta_n$ by Claim 1 and  $\zeta_n\leq  \rho_{n}+1$ by Claim 2. This concludes the proof.
\end{proof}

\section{WLP for algebras associated to paths and cycles}\label{Section4}
In this section, we study the WLP for artinian algebras associated to paths and cycles. From now on, we always assume $\Char(\kk)=0$ and denote by $\ell$ the sum of variables in the polynomial ring we are working with.

\subsection{Paths}
The artinian algebra associated to $P_n$ is
$$A(P_n)=R/K,$$
where $K=(x_1^2,\ldots,x_n^2)+(x_1x_2,x_2x_3,\ldots,x_{n-1}x_n)\subset R=\kk[x_1,\ldots,x_n]$. The following lemma is useful to an inductive argument on the WLP of $A(P_n)$.

\begin{Lem}\label{Lemma3.12}
For every integer $i$, there is a commutative diagram with exact rows
\begin{align*}
\xymatrix{0\ar[r]& [A(P_{n-2})]_{i-1}\ar[r]\ar[d]_{\cdot\ell} & [A(P_n)]_i\ar[r] \ar[d]_{\cdot\ell}& [A(P_{n-1})]_i\ar[r]\ar[d]^{\cdot\ell}&0\\
		0\ar[r]& [A(P_{n-2})]_{i}\ar[r] & [A(P_n)]_{i+1}\ar[r]& [A(P_{n-1})]_{i+1}\ar[r]&0 }.
	\end{align*}
\end{Lem}
\begin{proof} Assume $A(P_n)=R/K$ and set $I=K+(x_n)$ and $J=(K\colon x_n)$. Then
	$A(P_{n-1})\cong R/I$ and $A(P_{n-2})\cong R/J$ and we have the following exact sequence
	$$\xymatrix {0\ar[r]& R/J (-1)\ar[r]^{\quad \cdot x_n}&R/K\ar[r]&R/I\ar[r]&0.}$$
The desired conclusion follows.
\end{proof}
 We now prove our first main result.
\begin{Theo}\label{thm_WLP_paths}
The ring  $A(P_n)$ has the WLP if and only if $n\in \{1,2,\ldots,7,9,10,13\}.$ 
\end{Theo}

\begin{proof}
Using {\tt Macaulay2} \cite{M2codes} to compute the Hilbert series of the rings $A(P_n)$ and $A(P_n)/\ell A(P_n)$ with $1\leq n\leq 17$, we see that $A(P_n)$ has the WLP for each $n\in\{1,2,\ldots,7,9,10,13\}.$ Furthermore, for each $n\in \{8,11,14,15,17\},\ A(P_n)$ only fails the surjectivity in the multiplication map by $\ell$ from degree $\lambda_n$ to degree $\lambda_n+1$. However, for $n\in\{12,16\}, \ A(P_n)$ only fails the injectivity in the multiplication map by $\ell$ from degree $\lambda_n-1$ to degree $\lambda_{n}$.
	
It remains to show the following 
	
\noindent \textsc{Claim}: The multiplication map $\cdot \ell: [A(P_n)]_{\lambda_n}\longrightarrow [A(P_n)]_{\lambda_n+1}$ is not surjective for all $n\geq 17$.
	
We will prove the above claim by induction on $n$, having just established the case $n=17$. For $n\geq 18$, we consider the multiplication map $$\cdot \ell: [A(P_n)]_{\lambda_n}\longrightarrow [A(P_n)]_{\lambda_n+1}.$$ By Lemma~\ref{lem_compare_lambda}, one has $\lambda_{n-1}\le \lambda_n\le \lambda_{n-1}+1$, hence we consider the following two cases.
	
\noindent \underline{{\bf Case 1: $\lambda_{n}=\lambda_{n-1}$.}} In the diagram of Lemma~\ref{Lemma3.12} where $i=\lambda_{n-1}=\lambda_n$, the right vertical map is not surjective by the induction hypothesis, so neither is the middle vertical map.
	
\noindent \underline{{\bf Case 2: $\lambda_{n}=\lambda_{n-1}+1$.}} By Lemma~\ref{lem_compare_lambda}, one has $\lambda_{n-1}=\lambda_{n-2}=\lambda_{n-3}$. In this case, we must have $n\geq 20$ since $\lambda_{16}=\lambda_{17}=\lambda_{18}=\lambda_{19}=5$.

Assume $A(P_n)=R/K$ and set $I=K+(x_{n-2})$ and $J=K\colon x_{n-2}$. Then we have the following exact sequence
$$\xymatrix {0\ar[r]& R/J (-1)\ar[rr]^{\quad \cdot x_{n-2}}&&R/K\ar[r]&R/I\ar[r]&0},$$
where $R/J\cong A(P_{n-4})\otimes_\kk \kk[x_n]/(x_n^2)$ and $R/I\cong A(P_{n-3})\otimes_\kk A(P_2)$, with $A(P_2)=\kk[x_{n-1},x_n]/(x_{n-1}, x_n)^2.$ This exact sequence gives rise to the following commutative diagram, with exact rows
\begin{align*}
\xymatrix{0\ar[r]& [R/J]_{\lambda_n-1}\ar[r]\ar[d]_{\cdot\ell} & [A(P_n)]_{\lambda_n}\ar[r] \ar[d]_{\cdot\ell}& [R/I]_{\lambda_{n}}\ar[r]\ar[d]^{\cdot\ell}&0\\
0\ar[r]& [R/J]_{\lambda_n}\ar[r] & [A(P_n)]_{\lambda_n+1}\ar[r]& [R/I]_{\lambda_n+1}\ar[r]&0}
\end{align*}
To prove that the middle vertical map is not surjective, it suffices to show that the right vertical map
\[
\cdot\ell: [R/I]_{\lambda_{n}}\longrightarrow [R/I]_{\lambda_{n}+1}
\]
is not surjective. By the inductive hypothesis, $A(P_{n-3})$ fails the surjectivity from degree $\lambda_{n}-1$ to degree $\lambda_n$, as $\lambda_{n-3}=\lambda_n-1$. Clearly, the Hilbert function of $A(P_2)$ is $(1,2)$, and hence $A(P_2)$ fails the surjectivity from degree 0 to degree 1. Then by Lemma~\ref{lem_tensor}(a), $R/I\cong A(P_{n-3})\otimes_\kk A(P_2)$ fails the surjectivity from degree $\lambda_n$ to degree $\lambda_n+1$, as desired.
\end{proof}

The above theorem shows that $A(P_n)$ fails the WLP  since surjectivity fails for any $n\geq 17$. The next result also prove that $A(P_n)$ fails the injectivity for some cases.

\begin{pro}\label{Fails_injectivityofPaths}
Recall the mode $\lambda_n$ of the independence polynomial of $I(P_n;t)$.  If  $n\geq 12$  is an integer such that  $\lambda_n=\lambda_{n-1}+1,$ then $A(P_n)$ fails the injectivity from degree $\lambda_n -1$ to degree $\lambda_n$. 
\end{pro}
\begin{proof}
We prove the above proposition by induction on $n\geq 12$. A computation with {\tt Macaulay2} \cite{M2codes} shows that the proposition holds for $n\in\{12,16,20\}.$ This covers all cases from $12$ to $20$ due to Lemma~\ref{lem_compare_lambda}. Now consider an $n\geq 21$ such that $\lambda_n=\lambda_{n-1}+1.$ Set
\begin{align*}
n_1&=\max\{j \mid j<n \; \text{and}\; \lambda_j=\lambda_{j-1}+1\}\\
n_2&=\max\{j \mid j<n_1 \; \text{and}\; \lambda_j=\lambda_{j-1}+1\}\\
m&=\max\{j \mid j<n_2 \; \text{and}\; \lambda_j=\lambda_{j-1}+1\}.
\end{align*}
Then, by Lemma~\ref{lem_compare_lambda}, $9\le n-m\le 11$. We have the following exact sequence
\[
0\to A(P_m)\otimes_\kk A(P_{n-m-3}) (-1)\xrightarrow{\cdot x_{m+2}}  A(P_n)\to A(P_{m+1})\otimes_\kk A(P_{n-m-2})\to 0.
\]
By using this exact sequence, it suffices to show that 
$$\cdot \ell: [A(P_m)\otimes_\kk A(P_{n-m-3})]_{\lambda_n-2}\longrightarrow [A(P_m)\otimes_\kk A(P_{n-m-3})]_{\lambda_n-1}$$
is not injective. By the inductive hypothesis, $A(P_m)$ fails the injectivity from degree $\lambda_m-1$ to $\lambda_m$. Observe that $\lambda_m=\lambda_n-3$ and $6\le n-m-3\leq 8$. Hence by Table \ref{tab_indpoly}, $\lambda_{n-m-3}= 2$ and consequently, $A(P_{n-m-3})$ fails the injectivity from degree 2 to degree 3. By Lemma~\ref{lem_tensor}(b), $A(P_m)\otimes_\kk A(P_{n-m-3})$ fails the injectivity from degree $\lambda_m+1=\lambda_n-2$ to $\lambda_n-1$, as desired.
\end{proof}

\subsection{Cycles}
The artinian algebra associated to the cycle on $n$ vertices is
$$A(C_n)=R/K,$$
where $K=(x_1^2,\ldots,x_n^2)+(x_1x_2,x_2x_3,\ldots,x_{n-1}x_n,x_nx_1)\subset R=\kk[x_1,\ldots,x_n]$. 
Our second main result is the following.
\begin{Theo}\label{thm_WLP_cycles}
The algebra $A(C_n)$ has the WLP if and only if $n\in \{3,4,\ldots,11,13,\\14,17\}.$ 
\end{Theo}
\begin{proof}
	Recall that $\rho_n$ is the mode of the independence polynomial of  $C_n$. Using {\tt Macaulay2} \cite{M2codes} to compute the Hilbert series of $A(C_n)$ and $A(C_n)/\ell A(C_n)$ with $3\leq n\leq 20$, we can check that:
	\begin{itemize}
		\item $A(C_n)$ has the WLP for each $3\leq n\leq 17$ and $n\notin\{12,15,16\}$; 
		\item for $n \in  \{12,15,18,19\}$, then $A(C_n)$ fails the surjectivity from degree $\rho_n$ to degree $\rho_{n}+1$;
		\item for $n \in  \{16,20\}$, then $A(C_n)$ fails the injectivity from degree $\rho_n -1$ to degree $\rho_n.$  
	\end{itemize}

Now assume that $n\geq 21$. By Lemmas \ref{lem_compare_lambda} and \ref{lem_compare_rho_lambda},
$\lambda_{n-1}\leq \rho_n\leq \lambda_{n-4}+1 \le \lambda_{n-1}+1$.  Consider the following two cases.

\noindent \underline{{\bf Case 1: $\rho_n=\lambda_{n-1}$.}} In this case, we will show that $A(C_n)$ fails the WLP due to the failure of the surjectivity from degree $\rho_n$ to degree $\rho_{n}+1.$ Indeed, write $A(C_n)=R/K$, and let $I=K+(x_n)$ and $J=K\colon x_n$.  Then
$A(P_{n-1})\cong R/I$ and $A(P_{n-3})\cong R/J$ and we have the following exact sequence
$$\xymatrix {0\ar[r]& R/J(-1)\ar[r]^{\quad \cdot x_n}&R/K \ar[r]&R/I\ar[r]&0.}$$
This yields a commutative diagram
	\begin{align*}
	\xymatrix{0\ar[r]& [A(P_{n-3})]_{\rho_n-1}\ar[r]\ar[d]_{\cdot\ell} & [A(C_n)]_{\rho_n}\ar[r] \ar[d]_{\cdot\ell}& [A(P_{n-1})]_{\rho_n}\ar[r]\ar[d]^{\cdot\ell}&0\\
		0\ar[r]& [A(P_{n-3})]_{\rho_n}\ar[r] & [A(C_n)]_{\rho_n+1}\ar[r]& [A(P_{n-1})]_{\rho_n+1}\ar[r]&0 }
	\end{align*}
	The proof of Theorem~\ref{thm_WLP_paths} shows that the multiplication map 
	$$\cdot\ell: [A(P_{n-1})]_{\rho_n}\longrightarrow [A(P_{n-1})]_{\rho_n+1}$$ is not surjective for any $n\geq 18.$ Hence the middle vertical map 
\[
\cdot\ell: [A(C_{n})]_{\rho_n}\longrightarrow [A(C_{n})]_{\rho_n+1}
\] is not surjective, as desired.
	
	\noindent \underline{{\bf Case 2: $\rho_n=\lambda_{n-1}+1$.}} In this case, Lemmas \ref{lem_compare_lambda} and \ref{lem_compare_rho_lambda} yield $\lambda_{n-1}=\lambda_{n-4}$.
	
Denote $y_1=x_{n-1}, y_2=x_{n-2}$. We have the following diagram 
\begin{align*}
\xymatrix{[A(C_n)]_{\rho_n}\ar@{->>}[r]^{/(x_n)\quad} \ar[d]_{\cdot \ell} & [A(P_{n-1})]_{\lambda_{n-1}+1}\ar@{->>}[r]^{/(x_{n-3})\qquad} & \left[A(P_{n-4}) \otimes_\kk \frac{\kk[y_1,y_2]}{(y_1,y_2)^2}\right]_{\lambda_{n-4}+1}\ar[d]^{\cdot \ell}\\
[A(C_n)]_{\rho_n+1}\ar@{->>}[r] & [A(P_{n-1})]_{\lambda_{n-1}+2}\ar@{->>}[r]& \left[A(P_{n-4}) \otimes_\kk \frac{\kk[y_1,y_2]}{(y_1,y_2)^2}\right]_{\lambda_{n-4}+2} }
\end{align*}
By the proof of Theorem \ref{thm_WLP_paths} and the fact that $n-4\ge 17$, the map 
$$A(P_{n-4}) \xrightarrow{\cdot \ell} A(P_{n-4})$$
fails the surjectivity at degree $\lambda_{n-4}$. Since the map 
\[
\kk[y_1,y_2]/(y_1,y_2)^2 \xrightarrow{\cdot (y_1+y_2)} \kk[y_1,y_2]/(y_1,y_2)^2
\] 
fails the surjectivity at degree 0, Lemma \ref{lem_tensor}(a) yields that the right vertical map of the diagram fails the surjectivity at degree $\lambda_{n-4}+1$.
	
By the surjectivity of the horizontal maps in the diagram, we conclude that left vertical map in the diagram fails the surjectivity at degree $\lambda_{n-4}+1=\rho_n$. Hence $A(C_n)$ does not have the WLP. This concludes the proof.
\end{proof}

Next, we consider a special case of the tadpole graphs. The \emph{tadpole} graph $T_{3,n}$ is obtained by joining a cycle $C_3$ to a path $P_n$  with a bridge. 

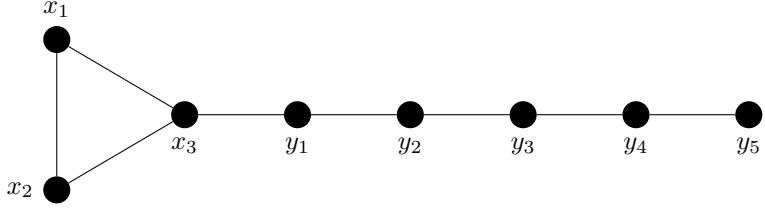
\begin{figure}[!h]
	\centering
	\begin{tikzpicture}
	[
	every edge/.style = {draw=black,very thick},
	vrtx/.style args = {#1/#2}{%
		circle, draw, thick, fill=black,
		minimum size=1mm, label=#1:#2}
	]
		\node (n1)[vrtx=above/$x_1$] at (-1.7,1) {};
		\node (n2) [vrtx=left/$x_2$] at (-1.7,-1)  {};
		\node (n3) [vrtx=below/$x_3$]at (0,0)  {};
		\node (n4) [vrtx=below/$y_1$] at (1.5,0) {};
		\node (n5) [vrtx=below/$y_2$] at (3,0)  {};
		\node (n6) [vrtx=below/$y_3$]at (4.5,0) {};
		\node (n7) [vrtx=below/$y_4$]at (6,0)  {};
		\node (n8) [vrtx=below/$y_5$]at (7.5,0)  {};
		\foreach \from/\to in {n1/n2,n1/n3,n2/n3,n3/n4,n4/n5,n5/n6,n6/n7,n7/n8}
		\draw (\from) -- (\to);	
	\end{tikzpicture}
	\caption{ Tadpole $T_{3,5}$}
\end{figure}

 Clearly, $T_{3,n}$ is a claw-free graph. Therefore, the independence polynomial of $T_{3,n}$ is unimodal \cite{Hamidoune90}. 
By Proposition~\ref{FormulaforIndependence} for either of the vertices on the left of the three cycle, we have
\begin{align*}
	I(T_{3,n};t)=I(P_{n+2};t)+tI(P_{n};t)=I(C_{n+3};t).
\end{align*}
By Proposition~\ref{Prop_independenceofCn}, it follows that the mode of $I(T_{3,n};t)$ is equal to that of $I(C_{n+3};t)$, which is 
$$
\rho_{n+3}=\left\lceil\frac{5(n+3)-4-\sqrt{5(n+3)^2-4}}{10}\right\rceil.
$$
\begin{Cor}\label{thm_WLP_tadpole(3,n)}
	The algebra  $A(T_{3,n})$ has the WLP if and only if $n\in \{1,3,4,7\}.$ 
\end{Cor}
\begin{proof}
The artinian algebra associated to $T_{3,n}$ is
\[
A(T_{3,n})=\frac{\kk[x_1,x_2,x_3,y_1,\ldots,y_{n}]}{(x_1^2,x_2^2,x_3^2,y_1^2,\ldots,y_{n}^2)+(x_1x_2,x_2x_3, x_3x_1, x_3y_1, y_1y_2,\ldots,y_{n-1}y_n}).
\]
Using {\tt Macaulay2} \cite{M2codes} to compute the Hilbert series of $A(T_{3,n})$ and $A(T_{3,n})/\ell A(T_{3,n})$ with $1\leq n\leq 17$, we can check that:
\begin{itemize}
	\item $A(T_{3,n})$ has the WLP for each  $n\in\{1,3,4, 7\}$; 
	\item for $n \in  \{2,5,8,9,11, 12, 14, 15, 16, 17\}$, then $A(T_{3,n})$ fails the surjectivity from degree $\rho_{n+3}$ to degree $\rho_{n+3}+1$;
	\item for $n \in  \{2, 6, 10, 13, 14, 17\}$, then $A(T_{3,n})$ fails the injectivity from degree $\rho_{n+3} -1$ to degree $\rho_{n+3}.$  
\end{itemize}
Now assume that $n\geq 18$. By Lemmas \ref{lem_compare_lambda} and \ref{lem_compare_rho_lambda}, 
$\lambda_{n+2}\leq \rho_{n+3}\leq \lambda_{n+2}+1$. We consider the following commutative diagram
\begin{align*}
\xymatrix{0\ar[r]& [A(P_{n})]_{\rho_{n+3}-1}\ar[r]^{ \cdot x_1\quad}\ar[d]_{\cdot\ell} & [A(T_{3,n})]_{\rho_{n+3}}\ar[r] \ar[d]_{\cdot\ell}& [A(P_{n+2})]_{\rho_{n+3}}\ar[r]\ar[d]^{\cdot\ell}&0\\
0\ar[r]& [A(P_{n})]_{\rho_{n+3}}\ar[r]^{\cdot x_1\quad} & [A(T_{3,n})]_{\rho_{n+3}+1}\ar[r]& [A(P_{n+2})]_{\rho_{n+3}+1}\ar[r]&0 }.
\end{align*}
The proof proceeds along the same lines as that of Theorem~\ref{thm_WLP_cycles}, replacing $A(C_{n+3})$ by $A(T_{3,n})$ and noting that $n+3\ge 21$.
\end{proof}

\subsection{Pans}
To study the WLP of rings associated to pans, we first examine the WLP of rings associated to $\CAE_n$. The latter is by definition
\[
A(\CAE_n)=\frac{\kk[x_1,\ldots,x_n]}{(x_1^2,\ldots,x_n^2)+(x_1x_2,x_2x_3,\ldots,x_{n-1}x_n,x_nx_1)+(x_{n-2}x_{n})}.
\]

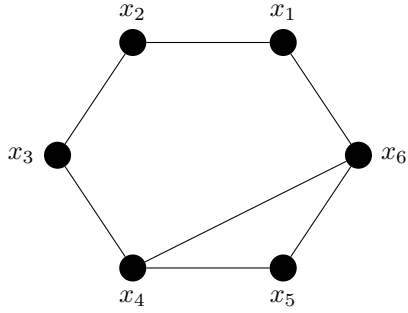
\begin{figure}[h]
	\centering
	\begin{tikzpicture}[
	every edge/.style = {draw=black,very thick},
	vrtx/.style args = {#1/#2}{%
		circle, draw, thick, fill=black,
		minimum size=1mm, label=#1:#2}
	]
	\node (n1) [vrtx=above/$x_2$]  at (-1,0) {};
	\node (n2) [vrtx=above/$x_1$]at (1,0)  {};
	\node (n3) [vrtx=right/$x_6$]at (2,-1.5)  {};
	\node (n4) [vrtx=below/$x_5$]at (1,-3) {};
	\node (n5) [vrtx=below/$x_4$]at (-1,-3)  {};
	\node (n6) [vrtx=left/$x_3$]at (-2,-1.5)  {};
	\foreach \from/\to in {n1/n2,n2/n3,n3/n4,n4/n5,n5/n6,n6/n1, n3/n5}
	\draw (\from) -- (\to);	
\end{tikzpicture}
\caption{The graph $\CAE_6$}
\end{figure}

\begin{Theo}
\label{thm_WLP_CE}
For an integer $n\ge 4$, the algebra $A(\CAE_n)$ has the WLP if and only if $n\in \{4,5,\ldots,8,10,11,14\}.$ 
\end{Theo}
\begin{proof}
By using {\tt Macaulay2} \cite{M2codes} to compute the Hilbert series of $A(\CAE_n)/\ell A(\CAE_n)$ and $A(\CAE_n)$ with $4\leq n\leq 20$, we can check that:
	\begin{itemize}
		\item $A(\CAE_n)$ has the WLP for each $4\leq n\leq 14$ and $n\notin\{9,12,13\}$; 
		\item for $n \in  \{9,12,15,16,18,19\}$,  $A(\CAE_n)$ fails the surjectivity from degree $\chi_n$ to degree $\chi_n+1$;
		\item for $n \in  \{13,17,20\}$, $A(\CAE_n)$ fails the injectivity from degree $\chi_n -1$ to degree $\chi_n.$  
	\end{itemize}
	Now assume that $n\geq 21$.  We will prove that $A(\CAE_n)$ fails the surjectivity from degree $\chi_n$ to degree $\chi_n +1$. The proof is similar to that of Theorem~\ref{thm_WLP_cycles}. Recall that  by Lemmas~\ref{lem_compare_lambda} and \ref{lem_compare_lambda_chi},
	\begin{equation}
	\label{eq_compare_chi_lambda}
	\lambda_{n-1}\leq \chi_n\leq \lambda_{n-4}+1 \le \lambda_{n-1}+1.
	\end{equation}
We consider the following two cases.

\noindent \underline{{\bf Case 1: $\chi_n=\lambda_{n-1}$.}} Consider the exact sequence
$$\xymatrix {0\ar[r]& A(P_{n-4}) (-1)\ar[rr]^{\quad \cdot x_{n}}&&A(\CAE_n)\ar[r]&A(P_{n-1})\ar[r]&0}.$$
The proof of Theorem~\ref{thm_WLP_paths} shows that the multiplication map 
	$$\cdot\ell: [A(P_{n-1})]_{\chi_n}\longrightarrow [A(P_{n-1})]_{\chi_n+1}$$ is not surjective for any $n\geq 18.$ Hence, the  map 
	\[
	\cdot\ell: [A(\CAE_{n})]_{\chi_n}\longrightarrow [A(\CAE_{n})]_{\chi_n+1}
	\] is also not surjective, as desired.

\noindent \underline{{\bf Case 2: $\chi_n = \lambda_{n-1}+1$.}} In this case, the chain \eqref{eq_compare_chi_lambda} yields $\lambda_{n-1}=\lambda_{n-4}$. As in the proof of Theorem~\ref{thm_WLP_cycles}, denoting $y_1=x_{n-1}, y_2=x_{n-2}$, we have the following diagram
	\begin{align*}
	\xymatrix{[A(\CAE_n)]_{\chi_n}\ar@{->>}[r]^{/(x_{n})\quad} \ar[d]_{\cdot \ell} & [A(P_{n-1})]_{\lambda_{n-1}+1}\ar@{->>}[r]^{/(x_{n-3})\quad\qquad} & \left[A(P_{n-4}) \otimes_\kk \frac{\kk[y_1,y_2]}{(y_1,y_2)^2}\right]_{\lambda_{n-4}+1}\ar[d]^{\cdot \ell}\\
		[A(\CAE_n)]_{\chi_n+1}\ar@{->>}[r] & [A(P_{n-1})]_{\lambda_{n-1}+2}\ar@{->>}[r]& \left[A(P_{n-4}) \otimes_\kk \frac{\kk[y_1,y_2]}{(y_1,y_2)^2}\right]_{\lambda_{n-4}+2}}
	\end{align*}
Since the right vertical map of the diagram fails the surjectivity at degree $\lambda_{n-4}+1$, we conclude that left vertical map in the diagram fails the surjectivity at degree $\lambda_{n-4}+1=\chi_n$, as desired.
\end{proof}

Finally, we show the last main result.
\begin{Theo}\label{thm_WLP_pan}
The algebra $A(\Pan_n)$ associated to the pan graph $\Pan_n$ has the WLP if and only if $n\in \{3,4,\ldots,10,12,13,16\}.$ 
\end{Theo}
\begin{proof}

\begin{figure}[!h]
\centering
		\begin{tikzpicture}[
			every edge/.style = {draw=black,very thick},
			vrtx/.style args = {#1/#2}{
				circle, draw, thick, fill=black,
				minimum size=1mm, label=#1:#2}
			]
			\node (n1) [vrtx=above/$x_2$]  at (-1,0) {};
			\node (n2) [vrtx=above/$x_1$]at (1,0)  {};
			\node (n3) [vrtx=below/$x_6$]at (2,-1.5)  {};
            \node (n4) [vrtx=below/$x_5$]at (1,-3) {};
			\node (n5) [vrtx=below/$x_4$]at (-1,-3)  {};
			\node (n6) [vrtx=left/$x_3$]at (-2,-1.5)  {};
			\node (n7) [vrtx=below/$x_7$]at (4,-1.5)  {};
			\foreach \from/\to in {n1/n2,n2/n3,n3/n4,n4/n5,n5/n6,n6/n1, n3/n7}		
			\draw (\from) -- (\to);		
		\end{tikzpicture}
	\caption{$\Pan_6$}
\end{figure}

The artinian algebra associated to $\Pan_n$ is
\[
A(\Pan_n)=\frac{\kk[x_1,\ldots,x_{n+1}]}{(x_1^2,\ldots,x_{n+1}^2)+(x_1x_2,x_2x_3,\ldots,x_{n-1}x_n,x_nx_{n+1})+(x_1x_n)}.
\]
By using {\tt Macaulay2} \cite{M2codes} to compute the Hilbert series of $A(\Pan_n)$ and $\dfrac{A(\Pan_n)}{\ell A(\Pan_n)}$ with $3\leq n\leq 20$, we can check that:
	\begin{itemize}
		\item $A(\Pan_n)$ has the WLP for each $3\leq n\leq 16$ and $n\notin\{11,14,15\}$; 
		\item for $n \in  \{11,14,17,18,20\}$, $A(\Pan_n)$ fails the surjectivity from degree $\zeta_n$ to degree $\zeta_n+1$;
		\item for $n \in  \{15,19\}$, $A(\Pan_n)$ fails the injectivity from degree $\zeta_n -1$ to degree $\zeta_n.$  
	\end{itemize}
	Now assume that $n\geq 21$. Recall that by Lemma \ref{lem_compare_lambda_chi-zeta},
	\begin{equation}
	 \label{eq_bounding_chi}
	 \chi_{n+1}\leq \zeta_n\leq \rho_{n}+1 \le \lambda_{n}+1\leq \chi_{n+1}+1.
	\end{equation}
	We consider the following two cases.

\noindent \underline{{\bf Case 1: $\zeta_n=\chi_{n+1}$.}} In this case, $A(\Pan_n)$ fails the surjectivity from degree $\zeta_n$ to degree $\zeta_n+1$  by using the exact sequence
$$\xymatrix {0\ar[r]& A(P_{n-3}) (-2)\ar[rr]^{\quad \cdot x_1x_{n+1}}&&A(\Pan_n)\ar[r]&A(\CAE_{n+1})\ar[r]&0}.$$
Indeed, the proof of Theorem~\ref{thm_WLP_CE} shows that the multiplication map 
$$\cdot\ell: [A(\CAE_{n+1})]_{\zeta_n}\longrightarrow [A(\CAE_{n+1})]_{\zeta_n+1}$$ is not surjective for any $n\geq 21.$ Hence the map 
\[
\cdot\ell: [A(\Pan_n)]_{\zeta_n}\longrightarrow [A(\Pan_n)]_{\zeta_n+1}
\] is also not surjective, as desired.
	
	\noindent \underline{{\bf Case 2: $\zeta_n=\chi_{n+1}+1$.}} In this case, the chain \eqref{eq_bounding_chi} yields $\lambda_{n}=\rho_{n}=\zeta_n-1$. Since $\lambda_n-\lambda_{n-3}\leq 1$ by Lemma \ref{lem_compare_lambda}, we have the following two subcases.

	\noindent \underline{{\bf Subcase 2.1: $\lambda_n=\lambda_{n-3}$.}} As in the proof of Theorem~\ref{thm_WLP_cycles}, denote $y_1=x_{n}, y_2=x_{n+1}$, we have the following diagram
	\begin{align*}
	\xymatrix{[A(\Pan_n)]_{\zeta_n}\ar@{->>}[rr]^{/(x_{n-1})\quad} \ar[d]_{\cdot \ell} && [A(P_{n})]_{\lambda_{n}+1}\ar@{->>}[rr]^{/(x_{1})\qquad\qquad} && \left[A(P_{n-3}) \otimes_\kk \frac{\kk[y_1,y_2]}{(y_1,y_2)^2}\right]_{\lambda_{n-3}+1}\ar[d]^{\cdot \ell}\\
		[A(\Pan_n)]_{\zeta_n+1}\ar@{->>}[rr] && [A(P_{n})]_{\lambda_{n}+2}\ar@{->>}[rr]&& \left[A(P_{n-3}) \otimes_\kk \frac{\kk[y_1,y_2]}{(y_1,y_2)^2}\right]_{\lambda_{n-3}+2}}
	\end{align*}
	Since the right vertical map of the above diagram fails the surjectivity at degree $\lambda_{n-3}+1=\zeta_n$, we conclude that the left vertical map  fails the surjectivity at the same degree. 
	
	\noindent \underline{{\bf Subcase 2.2: $\lambda_n=\lambda_{n-3}+1$.}} Set
	$$m=\max\{j \mid j\leq n\; \text{and}\; \lambda_j=\lambda_{j-1}+1\}.$$
	Then by Lemma \ref{lem_compare_lambda}, $n-2\leq m\leq n$ and $\lambda_m=\lambda_n$. First, we consider the case where $m\ne n$. Set
	$$y=\begin{cases}
	x_{n-2}&\text{if}\; m=n-2\\
	x_{n+1}&\text{if}\; m=n-1.
	\end{cases} $$
	Then we have the following diagram
	\begin{align*}
	\xymatrix{0\ar[r]& [A(P_m)]_{\zeta_n-2}\ar[r]^{\cdot y}\ar[d]\ar[d]_{\cdot \ell} &[A(\Pan_n)]_{\zeta_n-1} \ar[d]^{\cdot \ell} \\ 0\ar[r]& [A(P_m)]_{\zeta_n-1}\ar[r]_{\cdot y} &[A(\Pan_n)]_{\zeta_n}  }
	\end{align*}
	Since $\zeta_n-2=\lambda_n-1=\lambda_m-1$, we have the first vertical map of the diagram fails the injectivity at degree $\lambda_m-1$ by Proposition~\ref{Fails_injectivityofPaths}. It follows that the second vertical map of the diagram fails the injectivity at degree $\zeta_n-1$. 
	
	To complete the proof of the theorem, we consider the case where $m=n$. In this case, one has $\rho_n=\lambda_n=\lambda_{n-1}+1$. By Lemmas~\ref{lem_compare_lambda} and~\ref{lem_compare_rho_lambda}, $\lambda_{n-1}=\lambda_{n-4}=\lambda_{n-5}+1$. Hence $\lambda_{n-4}=\zeta_n-2$. Now we consider the following diagram
	\begin{align*}
	\xymatrix{0\ar[r]& [A(P_{n-4})]_{\zeta_n-3}\ar[rr]^{\cdot x_{n-4}x_{n-2}}\ar[d]\ar[d]_{\cdot \ell} &&[A(\Pan_n)]_{\zeta_n-1} \ar[d]^{\cdot \ell} \\ 0\ar[r]& [A(P_{n-4})]_{\zeta_n-2}\ar[rr]_{\cdot x_{n-4}x_{n-2}} &&[A(\Pan_n)]_{\zeta_n}.  }
	\end{align*}
	By Proposition~\ref{Fails_injectivityofPaths}, the first vertical map of the diagram fails the injectivity at degree $\lambda_{n-4}-1=\zeta_n-3$. It follows that the second vertical map of the diagram fails the injectivity at degree $\zeta_n-1$. Thus we complete the proof.
\end{proof}

\section*{Acknowledgments}
It is our pleasure  to thank Prof.\,Le Tuan Hoa for his constant encouragement during the preparation of this work.  The first named author was partially supported by the Vietnam Academy of Science and Technology (through the grant NCXS02.01/22-23). The second named author was partially supported by the NAFOSTED (Vietnam) under the grant number 101.04-2023.07. Parts of the paper were written during our stay at the Vietnam Institute for Advanced Study in Mathematics (VIASM). We would like to thank  the VIASM for its generous support and hospitality. Finally, the authors are also grateful to two anonymous referees for their careful reading of this manuscript and for many critical suggestions.

\bibliographystyle{plain} 
\bibliography{WLP_Graphs} 

\begin{thebibliography}{10}

\bibitem{AEMS87}
Y.~Alavi, P.J. Malde, A.J. Schwenk, and P.~Erd\H{o}s.
\newblock The vertex independence sequence of a graph is not constrained.
\newblock {\em Congr. Numer.}, 58:15--23, 1987.

\bibitem{AB2020}
N.~Altafi and M.~Boij.
\newblock The weak {L}efschetz property of equigenerated monomial ideals.
\newblock {\em J. Algebra}, 556:136--168, 2020.

\bibitem{AN2018}
N.~Altafi and N.~Nemati.
\newblock Lefschetz properties of monomial algebras with almost linear
  resolution.
\newblock {\em Comm. Algebra}, 48(4):1499--1509, 2020.

\bibitem{BK2013}
A.~Bhattacharyya and J.~Kahn.
\newblock A bipartite graph with non-unimodal independent set sequence.
\newblock {\em Electron. J. Combin.}, 20(4):Paper 11, 3, 2013.

\bibitem{BMMNZ12}
M.~Boij, J.C. Migliore, R.M. Mir\'{o}-Roig, U.~Nagel, and F.~Zanello.
\newblock On the shape of a pure {$O$}-sequence.
\newblock {\em Mem. Amer. Math. Soc.}, 218(1024):viii+78, 2012.

\bibitem{BK2007}
H.~Brenner and A.~Kaid.
\newblock Syzygy bundles on {$\mathbb{P}^2$} and the weak {L}efschetz property.
\newblock {\em Illinois J. Math.}, 51(4):1299--1308, 2007.

\bibitem{BK2011}
H.~Brenner and A.~Kaid.
\newblock A note on the weak {L}efschetz property of monomial complete
  intersections in positive characteristic.
\newblock {\em Collect. Math.}, 62(1):85--93, 2011.

\bibitem{CookII2012}
D.~Cook~II.
\newblock The {L}efschetz properties of monomial complete intersections in
  positive characteristic.
\newblock {\em J. Algebra}, 369:42--58, 2012.

\bibitem{CN2011}
D.~Cook~II and U.~Nagel.
\newblock The weak {L}efschetz property, monomial ideals, and lozenges.
\newblock {\em Illinois J. Math.}, 55(1):377--395, 2011.

\bibitem{Cooperetal2023}
S.~M. Cooper, S.~Faridi, T.~Holleben, L.~Nicklasson, and A.V. Tuyl.
\newblock The weak {L}efschetz property of whiskered graphs.
\newblock {\em To apper in Lefschetz Properties: Current and New Directions,
  Springer INdAM series}, 2023.

\bibitem{DaoNair2022}
H.~Dao and R.~Nair.
\newblock On the lefschetz property for quotients by monomial ideals containing
  squares of variables.
\newblock {\em Communications in Algebra}, 52(3):1260--1270, 2024.

\bibitem{GLN2020}
O.~Gasanova, S.~Lundqvist, and L.~Nicklasson.
\newblock On decomposing monomial algebras with the {L}efschetz properties.
\newblock {\em J. Pure Appl. Algebra}, 226(6):Paper No. 106968, 15, 2022.

\bibitem{Macaulay2}
D.R. Grayson and M.E. Stillman.
\newblock Macaulay2, a software system for research in algebraic geometry.

\bibitem{GH83}
I.~Gutman and F.~Harary.
\newblock Generalizations of the matching polynomial.
\newblock {\em Utilitas Math.}, 24:97--106, 1983.

\bibitem{Hamidoune90}
Y.O. Hamidoune.
\newblock On the numbers of independent {$k$}-sets in a claw free graph.
\newblock {\em J. Combin. Theory Ser. B}, 50(2):241--244, 1990.

\bibitem{HSS2011}
B.~Harbourne, H.~Schenck, and A.~Seceleanu.
\newblock Inverse systems, {G}elfand-{T}setlin patterns and the weak
  {L}efschetz property.
\newblock {\em J. Lond. Math. Soc. (2)}, 84(3):712--730, 2011.

\bibitem{HMMNWW2013}
T.~Harima, T.~Maeno, H.~Morita, Y.~Numata, A.~Wachi, and J.~Watanabe.
\newblock {\em The {L}efschetz properties}, volume 2080 of {\em Lecture Notes
  in Mathematics}.
\newblock Springer, Heidelberg, 2013.

\bibitem{HL94}
C.~Hoede and X.L. Li.
\newblock Clique polynomials and independent set polynomials of graphs.
\newblock {\em Discrete Math.}, 125:219--228, 1994.

\bibitem{GL84}
G.~Hopkins and W.~Staton.
\newblock Some identities arising from the {F}ibonacci numbers of certain
  graphs.
\newblock {\em Fibonacci Quart.}, 22(3):255--258, 1984.

\bibitem{KV2011}
A.R. Kustin and A.~Vraciu.
\newblock The weak {L}efschetz property for monomial complete intersection in
  positive characteristic.
\newblock {\em Trans. Amer. Math. Soc.}, 366, 09 2014.

\bibitem{LF2010}
J.~Li and F.~Zanello.
\newblock Monomial complete intersections, the weak {L}efschetz property and
  plane partitions.
\newblock {\em Discrete Math.}, 310(24):3558--3570, 2010.

\bibitem{MMO2013}
E.~Mezzetti, R.M. Mir\'{o}-Roig, and G.~Ottaviani.
\newblock Laplace equations and the weak {L}efschetz property.
\newblock {\em Canad. J. Math.}, 65(3):634--654, 2013.

\bibitem{MM2016}
M.~Micha\l{ek} and R.M. Mir\'{o}-Roig.
\newblock Smooth monomial {T}ogliatti systems of cubics.
\newblock {\em J. Combin. Theory Ser. A}, 143:66--87, 2016.

\bibitem{MNS2020}
J.~Migliore, U.~Nagel, and H.~Schenck.
\newblock The weak {L}efschetz property for quotients by quadratic monomials.
\newblock {\em Math. Scand.}, 126(1):41--60, 2020.

\bibitem{MMR03}
J.C. Migliore and R.M. Mir\'{o}-Roig.
\newblock Ideals of general forms and the ubiquity of the weak {L}efschetz
  property.
\newblock {\em J. Pure Appl. Algebra}, 182(1):79--107, 2003.

\bibitem{MMR17}
J.C. Migliore and R.M. Mir\'{o}-Roig.
\newblock On the strong {L}efschetz problem for uniform powers of general
  linear forms in {$k[x,y,z]$}.
\newblock {\em Proc. Amer. Math. Soc.}, 146(2):507--523, 2018.

\bibitem{MMN2011}
J.C. Migliore, R.M. Mir\'{o}-Roig, and U.~Nagel.
\newblock Monomial ideals, almost complete intersections and the weak
  {L}efschetz property.
\newblock {\em Trans. Amer. Math. Soc.}, 363(1):229--257, 2011.

\bibitem{MMN2012}
J.C. Migliore, R.M. Mir\'{o}-Roig, and U.~Nagel.
\newblock On the weak {L}efschetz property for powers of linear forms.
\newblock {\em Algebra Number Theory}, 6(3):487--526, 2012.

\bibitem{MN2013}
J.C. Migliore and U.~Nagel.
\newblock Survey article: a tour of the weak and strong {L}efschetz properties.
\newblock {\em J. Commut. Algebra}, 5(3):329--358, 2013.

\bibitem{MiroRoig2016}
R.M. Mir\'{o}-Roig.
\newblock Harbourne, {S}chenck and {S}eceleanu's conjecture.
\newblock {\em J. Algebra}, 462:54--66, 2016.

\bibitem{MRT19}
R.M. Mir\'{o}-Roig and Q.H. Tran.
\newblock On the weak {L}efschetz property for almost complete intersections
  generated by uniform powers of general linear forms.
\newblock {\em J. Algebra}, 551:209--231, 2020.

\bibitem{MRT2019}
R.M. Mir\'{o}-Roig and Q.H. Tran.
\newblock The weak {L}efschetz property for {A}rtinian {G}orenstein algebras of
  codimension three.
\newblock {\em J. Pure Appl. Algebra}, 224(7):106305, 2020.

\bibitem{MRT2020}
R.M. Mir\'{o}-Roig and Q.H. Tran.
\newblock The weak {L}efschetz property of {G}orenstein algebras of codimension
  three associated to the {A}p\'{e}ry sets.
\newblock {\em Linear Algebra Appl.}, 604:346--369, 2020.

\bibitem{M2codes}
Hop~D. Nguyen and Q.H. Tran.
\newblock Macaulay2 codes for checking the weak {L}efschetz property of
  artinian algebra associated to graphs.
\newblock Available at
  https://sites.google.com/view/tranquanghoassite/software?authuser=0.

\bibitem{Nick2018}
L.~Nicklasson.
\newblock The strong {L}efschetz property of monomial complete intersections in
  two variables.
\newblock {\em Collect. Math.}, 69(3):359--375, 2018.

\bibitem{HT2023}
H.V.N. Phuong and Q.H. Tran.
\newblock A new proof of {S}tanley's theorem on the strong {L}efschetz
  property.
\newblock {\em Colloq. Math.}, 173(1):1--8, 2023.

\bibitem{Stanley1980}
R.P. Stanley.
\newblock Weyl groups, the hard {L}efschetz theorem, and the {S}perner
  property.
\newblock {\em SIAM J. Algebraic Discrete Methods}, 1(2):168--184, 1980.

\bibitem{QHT2020}
Q.H. Tran.
\newblock The {L}efschetz properties of artinian monomial algebras associated
  to some graphs.
\newblock {\em Journal of Science, Hue University of Education}, 59(3):12--22,
  2021.

\end{thebibliography}

\end{document}